\documentclass{amsart}

\usepackage[colorlinks=true]{hyperref}
\hypersetup{urlcolor=blue, citecolor=red}

\usepackage{amsmath,amsthm,amssymb}
\usepackage{enumerate}

\newtheorem{theorem}{Theorem}[section]
\newtheorem{proposition}[theorem]{Proposition}

\newtheorem{remark}[theorem]{Remark}
\newtheorem{definition}[theorem]{Definition}

\numberwithin{equation}{section}

\begin{document}

\title{A remark on weak tracial approximation}


\author{Xiaochun Fang}
\address{School of Mathematical sciences, Tongji University, Shanghai, 200092, China}
\curraddr{}
\email{xfang@tongji.edu.cn}
\thanks{}

\author{Junqi Yang*}
\address{School of Mathematical sciences, Tongji University, Shanghai, 200092, China}
\curraddr{}
\email{yangjq24@tongji.edu.cn}
\thanks{}



\date{}

\dedicatory{}

\begin{abstract}
	In this paper, we point out that the definition of weak tracial approximation can be improved and strengthened.
	An example of weak tracial approximation is also provided.
\end{abstract}

\maketitle

\section{Introduction}

Before Elliott and Niu in \cite{EN2008tracial} introduced the notion of tracial approximation
by abstract classes of unital $C^*$-algebras,
tracially approximately finite-dimensional (TAF) algebras
and tracially approximately interval (TAI) algebras is considered by Lin
\cite{Lin2001tracial}, \cite{Lin2001tracially}.
The idea here is that if $A$ can be well-approximated by well-behaved algebras in trace,
then we could expect that $A$ is well-behaved too.
Many properties can be inherited from the given class to tracially approximated $C^*$-algebras.
In \cite{EN2008tracial}, finiteness (stably finiteness), being stable rank one, 
having nonempty tracial state space,
the property that the strict order on projections is determined by traces,
the property that any state on the $K_0$ group comes from a tracial state
is considered.
Elliott, Fan and Fang in \cite{EFF2018certain} investigated the inherence of several comparison properties and divisibility.

Phillips introduced large subalgebras in \cite{Phillips2014Large},
which is an abstraction of Putnam's orbit breaking subalgebra of the crossed product
of the Cantor set by a minimal homeomorphism in \cite{Putnam1989Cstar}.
In \cite{Phillips2014Large}, Phillips studied properties such as stably finite,
pure infinite,
the restriction map between trace state spaces,
and the relationship between their radius comparison.
Later Archey and Phillips in \cite{AP2020permanence} introduced the centrally large subalgebras
and proved that the property stable rank one can be inherited.
Zhao, Fang and Fan proved that the property real rank zero can also be inherited
in \cite{ZFF2019Permanence}.
In \cite{ABP2016Centrally}, Archey, Buck and Phillips studied the property of tracial $\mathcal{Z}$-absorbing when $A$ is stably finite.
Fan, Fang and Zhao investigated several comparison properties in \cite{FFZ2019Comparison},
and they also studied the inheritance of the local weak comparison \cite{ZFF2020Permanence},
the almost divisibility \cite{FFZ2021divisibility}
and the inheritance of tracial nuclear dimension
for centrally large subalgebras \cite{ZFF2020inherence}.

Recently Elliott, Fan and Fang are trying to put tracial approximation and large subalgebras together, to find the common inheritance properties
and offer a more flexible framework than tracial approximation and large subalgebras.
Tracially $\mathcal{Z}$-absorbing,
nuclear dimension, property (SP) and $m$-almost divisibility have been studied 
in their work \cite{EFF2022Generalized} by far.

In this paper, we point out that the definition of weak tracial approximation can be improved and strength.
A known example of weak tracial approximation is proved in Section.~\ref{example_sec}.

\section{Preliminaries}

Elliott, Fan and Fang introduced a class of generalized tracial approximation $C^*$-algebras in \cite[Def.~1.1]{EFF2022Generalized}.
They defined as follows the class of $C^*$-algebras which can be weakly tracially approximated by $C^*$-algebras in $\Omega$ and denote this class by $\mathrm{WTA\Omega}$.

\begin{definition}\label{def_WTAOmega}
	Let $\Omega$ be a class of unital $C^*$-algebras.
	A simple unital $C^*$-algebra $A$ is said to belong to the class $\mathrm{WTA\Omega}$ if,
	for any $\epsilon > 0$, any finite subset $F \subset A$, and any non-zero element $a \in A_+$,
	there exists a $C^*$-subalgebra $B \subset A$ with $B \in \Omega$,
	a projection $p \in A$ with $p = 1_B$,
	an element $g \in B_+$ with $0 \leq g \leq 1$, such that
	\begin{enumerate}[(1)]
		\item $(p-g)x \in_{\epsilon} B$, $x(p-g) \in_{\epsilon} B$ for all $x \in F$,
		\item $\| (p-g)x - x(p-g)\| < \epsilon$ for all $x \in F$,
		\item $1_A - (p-g) \precsim_A a$, and
		\item $\| (p-g)a(p-g)\| \geq \|a\| - \epsilon$.
	\end{enumerate}
\end{definition}

Note that $p-g$ always appears as a whole in the definition above.
We use a new variable $r = p-g$ and a different notation ${\rm GTA}\Omega$ to represent possible
non-unital $\Omega$ or non-unital $A$.
The idea of condition (3') comes from \cite[Def.~1.1]{FG2017weak}.

\begin{definition}\label{def_GTAOmega}
	Let $\Omega$ be a class of $C^*$-algebras.
	A simple $C^*$-algebra $A$ is said to belong to the class $\mathrm{GTA\Omega}$ if,
	for any $\epsilon > 0$, any finite subset $F \subset A$, and any positive elements $a \in A_+$ with $\|a\| = 1$
	and $b \in A_+$ with $\|b\| \leq 1$,
	there exists a $C^*$-subalgebra $B \subset A$ with $B \in \Omega$,
	an element $r \in B_+$ with $\|r\| \leq 1$, such that
	\begin{enumerate}[(1')]
		\item $rx \in_{\epsilon} B$, $xr \in_{\epsilon} B$ for all $x \in F$,
		\item $\| rx - xr\| < \epsilon$ for all $x \in F$,
		\item $(b^2 - brb -\epsilon)_+ \precsim_A a$, and
		\item $\| rar\| \geq 1 - \epsilon$.
	\end{enumerate}
\end{definition}

\begin{proposition}
	If $A$ and all algebras in $\Omega$ are unital,
	then $A \in {\rm WTA}\Omega$ is equivalent to $A \in {\rm GTA}\Omega$.
\end{proposition}

\begin{proof}
	If $A \in {\rm WTA}\Omega$, apply Definition~\ref{def_WTAOmega} with
	$(\epsilon, F, a)$ for any given $(\epsilon, F, a, b)$,
	getting corresponding $B, p$ and $g$.
	Then $r = p-g$ satisfies \begin{align*}
		(b^2 - brb - \epsilon)_+ & \precsim_A b(1 - r - \epsilon / \|b\|^2)_+ b
		\precsim_A (1 - r - \epsilon)_+ \precsim_A 1-r \precsim_A a \,,
	\end{align*}
	where the first $\precsim_A$ is from \cite[Lem.~2.3]{FG2017weak}.
	The rest condition is trivial.
	
	Now suppose that $A \in {\rm GTA}\Omega$.
	For any given $(\epsilon, F,a)$, set $\epsilon_1 = \epsilon / m$,
	where $m = 4M$ and $M = \max\{\|x\| : x \in F\} \vee 1$.
	Let \[
		f(t) = \begin{cases}
			0 , & \quad t \leq 0 ;\\
			\text{linear} , & \quad 0 \leq t \leq 1 - \epsilon_1 ;\\
			1 , & \quad t \geq 1 - \epsilon_1 \,.
		\end{cases}
	\]
	Then $|f(t) - t| \leq \epsilon_1$ for $t \in [0,1]$.
	Find $\delta > 0$ such that if $\|hc - ch \| < \delta$ for any $C^*$-algebra $C$ with
	$c,h \in C$, $\|c\| \leq M$ and $0 \leq h \leq 1$,
	then $\| f(h)c - cf(h)\|<\epsilon$.
	
	Applying Definition~\ref{def_GTAOmega} with $\epsilon' = \epsilon_1 \wedge \delta$,
	$b = 1_A$ and the same $F$ and $a$,
	we get a $C^*$-subalgebra $B \subset A$ in $\Omega$ and $r \in B$ with $0 \leq r \leq 1$
	such that $rx \in_{\epsilon / 2} B$.
	Combining with $rx \approx_{\epsilon_1} f(r)x$, we have $f(r)x \in_{\epsilon} B$ and similarly
	$xf(r) \in_{\epsilon} B$ for all $x \in F$.
	Note that $\| f(r)x - xf(r)\| < \epsilon$ is also true by our choice of $\delta$ and $\epsilon'$.
	
	We next prove condition (3) with $p = 1_B$ and $g = p - f(r)$.
	Note that \[
	1 - f(r) = \frac{1}{1 - \epsilon_1} (1 - r - \epsilon_1)_+
	\sim_A (1 - r - \epsilon_1)_+ \precsim_A (1 - r - \epsilon')_+ \precsim_A a \,,
	\]
	where the first $\precsim_A$ is from \cite[Lem.~2.7(i)]{KR2000nonsimple}
	and the last $\precsim_A$ is from condition (3').
	
	Finally, $f(r)af(r) \approx_{\epsilon_1} raf(r) \approx_{2\epsilon_1} rar$.
	So we have $\|f(r)af(r)\| \geq \| rar\| - 2 \epsilon_1 \geq 1 - 3 \epsilon_1 > 1 - \epsilon$.
\end{proof}

The following proposition showed that
when $A$ is finite, condition (4) is implied by condition (3) in Defintion~\ref{def_GTAOmega}.
This gives a proof that when a unital $C^*$-algebra $A$ is finite and $A \in {\rm TA}\Omega$,
then $A \in {\rm GTA}\Omega$.

\begin{proposition}
	Let $\Omega$ be a class of $C^*$-algebras.
	Let $A$ be a finite infinite dimensional simple unital $C^*$-algebra.
	Suppose that for any $\epsilon>0$,
	any finite set $F \subset A$, and any non-zero element $a \in A_+$,
	there exists a $C^*$-subalgebra $B \subset A$ with $B \in \Omega$,
	an element $r \in B_+$ with $\|r\| \leq 1$,
	such that \begin{enumerate}
		\item $rx \in_{\epsilon} B$, $xr \in_{\epsilon} B$ for all $x \in F$,
	\item $\| rx - xr\| < \epsilon$ for all $x \in F$, and
	\item $1_A - r \precsim_A a$.
	\end{enumerate}
	Then $A$ belongs to the class $\mathrm{GTA}\Omega$.
\end{proposition}

\begin{proof}
	Since $A$ is finite and unital, apply \cite[Lem.~2.9]{Phillips2014Large} to the given $A$, $a$ and $\epsilon$,
	obtaining a nonzero positive element $y \in \overline{aAa}$ such that
	whenever $g \precsim_A y$ for $g \in A$ with $0 \leq g \leq 1$,
	then $\|(1_A-g)a(1_A-g)\| > 1 - \epsilon$.
	
	Apply the hypothesis with $y$ in place of $a$ and everything else as given,
	getting $C^*$-subalgebra $B \subset A$ and $r \in B_+$ with $\|r\| \leq 1$.
	We have $1_A - r \precsim_A y \precsim_A a$ and
	$\|rar\| = \| (1_A - (1_A - r)) a (1_A - (1_A - r))\|> 1 - \epsilon$
	from the choice of $y$.
\end{proof}

The following proposition slightly strengthens condition (1) in Defintion~\ref{def_GTAOmega}.

\begin{proposition}\label{norm_cond}
	Let $\Omega$ be a class of $C^*$-algebras and $A$ be a unital $C^*$-algebra in $\mathrm{GTA\Omega}$.
	Then for any $m \in \mathbb{Z}_{>0}$, any $x_1,\cdots,x_m \in A$,
	any positive element $a \in A_+$ with $\|a\| = 1$,
	there exists $y_1, \cdots, y_m \in A$,
	a $C^*$-subalgebra $B \subset A$ with $B \in \Omega$,
	an element $r \in B_+$ with $\|r\| \leq 1$, such that
	\begin{enumerate}[(1)]
	\item $\|y_j - x_j\| < \epsilon$ for all $j = 1,\cdots,m$,
	\item $ry_j \in B$, $y_j r \in B$ for all $j = 1,\cdots,m$,
	\item $\| r x_j - x_j r\| < \epsilon$ for all $j = 1, \cdots, m$,
	\item $1_A - r \precsim_A a$, and
	\item $\| rar\| \geq 1 - \epsilon$.
	\end{enumerate}
\end{proposition}

\begin{proof}
	Let $1 > \epsilon > 0$.
	Let \[ f'(t) = \begin{cases}
			0, \quad & t \leq \frac{\epsilon}{3} ; \\
			\text{linear}, \quad & \frac{\epsilon}{3} \leq t \leq 1 - \frac{\epsilon}{3} ;\\
			1, \quad & t \geq 1 - \frac{\epsilon}{3} \,.
		\end{cases}
	\]
	So that $f' \in C_0(0,1]$, $f'(t) \approx_{\epsilon / 3} t$ and $1 - f'(t) = f'(1 - t)$ for $0 \leq t \leq 1$.
	Let $f_1 (t)$ be the unit of $f'(t)$ such that $f_1 f' = f' f_1 = f'$.
	For example one may take \[f_1(t) = \begin{cases}
			0, \quad & t \leq \frac{\epsilon}{6} \\
			\text{linear}, \quad & \frac{\epsilon}{6} \leq t \leq \frac{\epsilon}{3} \\
			1, \quad & t \geq \frac{\epsilon}{3} \,.
	\end{cases}\]
	Let $h \in C_0(0,1]$ such that $th(t) = f_1(t)$.
	Applying \cite[Lem.~2.5.11 (1)]{Lin2001introduction} to $\epsilon/3$ and $f(t) = (1 - f_1(t))^{1/2}$,
	there exists $\delta > 0$ such that for any elements $u \in A$ and $v \in A_+$ with $\|v\| \leq 1$,
	if $\|uv - vu\| < \delta$, then $\|uf(v) - f(v)u\| < \epsilon$.
	
	Apply the definition with $F = \{x_1, \cdots, x_m\}$ and
	with $\min\{\delta, \epsilon/3\}$ in place of $\epsilon$,
	getting $r_0 \in B_+$ with $\|r_0\| \leq 1$ such that $1_A - r_0 \precsim_A a$,
	and $\|r_0\|^2 \geq \| r_0 a r_0\| / \|a\| \geq 1 - \epsilon/3$.
	Let $r = f'(r_0) \in \mathrm{C}^*(r_0) \subset B_+$. Then $\|r\| \leq 1$, $r = f'(r_0) \precsim_A r_0$
	and $r = 1_A - f'(1_A - r_0)$.
	Therefore $1_A - r = f'(1_A - r_0) \precsim_A 1_A - r_0 \precsim_A a$.
	
	Since $r_0 x_j \in_{\epsilon/3} B$ and $r_0 x_j \approx_{\epsilon / 3} x_j r_0$,
	there exists $b_j \in B$ such that $b_j \approx_{\epsilon / 3} r_0 x_j$ and $b_j \approx_{2\epsilon/3} x_j r_0$.
	Set \[
		y_j = h(r_0) b_j + (1_A - f_1(r_0))^{1/2} x_j (1_A - f_1(r_0))^{1/2} \,. \]
	To verify $r y_j \in B$,
	note that $h(r_0)b_j \in B$ and $f'(r_0)(1_A - f_1(r_0)) = 0$.
	Similarly we have $y_j r \in B$.
	Condition (2) is verfied.
	
	To check $\| y_j - x_j\| < \epsilon$, note that
	\begin{align*}
		y_j & = h(r_0)b_j + (1_A - f_1(r_0))^{1/2} x_j (1_A - f_1(r_0))^{1/2} \\
		& \approx_{\epsilon / 3} f_1(r_0) x_j + (1_A - f_1(r_0))^{1/2} x_j (1_A - f_1(r_0))^{1/2} \\
		& \approx_{2\epsilon/3} f_1(r_0) x_j + (1_A - f_1(r_0)) x_j \\
		& \approx_{2\epsilon/3} x_j \,.
	\end{align*}
	Finnally, since $\| r - r_0\| < \epsilon /3$, we have \begin{align*}
		\| rar - r_0 a r_0\| & \leq \|rar - r_0 a r\| + \| r_0 a r - r_0 a r_0\| \\
		& \leq \| r - r_0\| \|a\| \|r\| + \|r_0\| \|a\| \|r - r_0\| \\
		& \leq \frac{2\epsilon}{3} \|a\| \,.
	\end{align*}
	Therefore $\|rar\| \geq \|r_0 a r_0\| - 2\epsilon / 3 \geq 1 - \epsilon$.
\end{proof}

\section{An Example}\label{example_sec}

Niu's tracial approximation theorem \cite[Thm.~3.9]{Niu2019radius}
provide us an example of generalized tracial approximation.
\cite{EFF2022Generalized} mentioned it after Definition 1.1 without an explicite
explaination.

\begin{proposition}
	Let $(X,\Gamma)$ be a free minimal topological dynamical system,
	where $X$ is a compact Hausdorff space and
	$\Gamma$ is an infinite countable discrete amenable group.
	If $(X,\Gamma)$ has the (URP) and (COS) in the sense of \cite[Def.~3.1, 4.1]{Niu2019radius},
	then the C*-algebra $A := C(X) \rtimes \Gamma$ belongs to the class ${\rm GTA}\Omega$,
	where \begin{equation*}
	\Omega = \Big\{
	\bigoplus_{s=1}^N \mathrm{M}_{K_s} (C_0(Z_s)) :
	\, N \in \mathbb{Z}_{>0},
	\text{each } Z_i \text{ is locally compact Hausdorff space }
	\Big\} \,.
	\end{equation*}
\end{proposition}

\begin{proof}
	Suppose that $(X,\Gamma)$ satisfies $(\lambda,m')$-(COS).
	For any finite subset $F = \{f_1 , f_2 , \cdots, f_n\}$ of $A$,
	$\epsilon > 0$ and $a \in A_+ \setminus \{0\}$,
	we always take $1_A$ to be the unit of $C^{\dag}$.
	By \cite[Lem.~4.7]{Niu2019radius}
	there exists mutually orthogonal nonzero elements
	$b_1 , \cdots , b_{m'} \in (\overline{aAa})_+$
	such that $b_1 \sim_A \cdots \sim_A b_{m'}$.
	By \cite[Lem.~4.2]{Niu2019radius},
	we obtain $b \in C(X)_+ \setminus \{0\}$ such that $b \precsim_A b_1$.
	
	Applying \cite[Thm.~3.9]{Niu2019radius} with $m=1$,
	$F = \{f_1 , \cdots , f_n\}$, $h = 1_A$,
	$$\delta = \min\{\epsilon/2 , \lambda \inf_{\mu} \mu(b^{-1}(0,\infty))\} \,,$$
	there are $p \in C(X)_+$ (not nesssarily a projection) with $\|p\| \leq 1$,
	a $C^*$-subalgebra $C \subset A$
	with $C \cong \bigoplus_{s=1}^N \mathrm{M}_{K_s}(C_0(Z_s))$
	and	$f'_1, \cdots , f'_n$ in $A$
	such that
	\begin{enumerate}[(i)]
		\item $\|f_i - f'_i\| < \delta$ for $i=1,\cdots,n$;
		\item $\| p f_i - f_i p\| < \delta$ for $i=1,\cdots,n$;
		\item $p \in C$, $p f'_i p \in C$ for $i=1,\cdots,n$;
		\item $\mu(X \setminus p^{-1}(1)) < \delta$
		for all $\mu \in \mathcal{M}_1(X,\Gamma)$.
	\end{enumerate}
	Let $B = C$ and $r = p^2$.
	Then condition (1') and (2') in Definition \ref{def_WTAOmega}
	can be easily deduced from (i), (ii) and (iii).
	Condition (iv) implies that
	$\mu ((1 - r)^{-1}(0,\infty)) < \lambda \mu(b^{-1}(0,\infty))$
	for all invariant Borel measure $\mu$.
	Therefore $1_A - r \precsim_A b \otimes 1_{m'}$ by $(\lambda,m')$-(COS).
	Thus $1_A - r \precsim_A b_1 \oplus \cdots \oplus b_{m'} \sim_A
	b_1 + \cdots + b_{m'} \precsim_A a$.
	This verifies condition (3').
	Condition (4') can be omitted when $A$ is finite (Prop.~\ref{norm_cond}).
	In fact $A$ has stable rank one by \cite[Thm.~7.8]{NL2020stable}.
\end{proof}

\begin{remark}
	It is true that $A$ belongs to the class ${\rm WTA}\Omega_1$,
	where $\Omega_1 = \{\tilde{C} : C \in \Omega\}$.
	In fact, if $A$ is unital and $A \in {\rm GTA}\Omega$,
	then $A \in {\rm WTA}\Omega_1$;
	if $A$ is non-unital and $A \in {\rm GTA}\Omega$,
	then $A^{\dagger} \in {\rm WTA}\Omega^\dagger$,
	where $\dagger$ means that we add a new unit even if $A$ or algebras in $\Omega$ are already unital.
\end{remark}

\end{document}